\documentclass[11pt]{amsart}

\usepackage{amsmath,amssymb,amsfonts,amsthm,mathrsfs,mathtools}
\usepackage{geometry}
\usepackage{hyperref}

\usepackage[activate={true,nocompatibility},final,
tracking=true,kerning=true,spacing=true,
factor=1100,stretch=10,shrink=10]{microtype}
            
\numberwithin{equation}{section}

\newtheorem{theorem}{Theorem}[section]
\newtheorem{lemma}[theorem]{Lemma}
\newtheorem{proposition}[theorem]{Proposition}

\newtheorem{conjecture}[theorem]{Conjecture}

\newtheorem{claim}{Claim}

\newcommand{\LM}{\mathscr{LM}}
\newcommand{\MM}{\mathscr{M}}
\newcommand{\calM}{\mathcal{M}}

\title{Two-place Laplacian matching root integral variations are impossible}

\author[S.M. Cioab\u{a}]{Sebastian M. Cioab\u{a}}
\address{Department of Mathematical Sciences, University of Delaware, Newark, DE 19716-2553, USA}
\email{cioaba@udel.edu}

\author[L. Liu]{LeLe Liu}
\address{School of Mathematical Sciences, Anhui University, Hefei 230601, Anhui, China}
\email{liu@ahu.edu.cn}

\author[Y. Wang]{Yi Wang}
\address{School of Mathematical Sciences, Anhui University, Hefei 230601, Anhui, China}
\email{wangy@ahu.edu.cn}

\thanks{{\it Corresponding author.} Yi Wang}
\thanks{{\it Funding.} Supported by the National Natural Science Foundation of China (No. 12331012, 12571360, 12471320)}
\date{\today}


\begin{document}

\begin{abstract}
Wang, Cui, and Cioab\u{a} introduced the Laplacian matching root integral variation 
of a graph and proved that it cannot occur in one place. They also showed that the 
two-place variation is impossible for connected graphs satisfying $g(G)/c(G)>7/6$, where $g(G)$ is the girth and $c(G)$ is the
dimension of the cycle space, and conjectured that no connected graph admits such a two-place variation. In this 
paper, we confirm this conjecture. The proof combines a structural relation obtained 
in their paper with two new power-sum identities for Laplacian matching roots.
\end{abstract}

\maketitle

\section{Introduction}

There are many graph polynomials whose roots encode structural information of the underlying graph such as the characteristic polynomial, the chromatic polynomial, the matching polynomial, and the Tutte polynomial (see \cite{EMM}). 

In this paper, we continue the study of a root integral variation problem for the Laplacian matching polynomial.

Throughout the paper, graphs are finite, simple, and undirected. Let $G$ be a graph with
vertex set $V(G)=\{v_1,\dots,v_n\}$ and edge set $E(G)$. For a vertex $v\in V(G)$, 
we denote by $N_G(v)$ the set of neighbors of $v$, and by $d_G(v)$ its degree. 
If $e$ is a non-edge of $G$, then $G+e$ denotes the graph obtained from $G$ by adding $e$.

For a graph $G$ of order $n$, the matching polynomial of $G$ is defined by
\begin{equation}\label{eq:matching}
\MM_G(x)=\sum_{M\in \calM(G)}(-1)^{|M|}x^{|V(G)\setminus V(M)|}
=\sum_{j=0}^{\lfloor n/2\rfloor}(-1)^j\phi_j(G)x^{n-2j},
\end{equation}
where $\calM(G)$ is the set of matchings of $G$ and $\phi_j(G)$ is the number of 
$j$-matchings. The matching polynomial was introduced by Heilmann and Lieb \cite{HL}. 
We refer to Godsil \cite{Godsil1} for further background on matching polynomials.

Motivated by the relation between the adjacency characteristic polynomial and the 
matching polynomial established by Godsil and Gutman \cite{GodsilGutman}, 
Mohammadian \cite{Moham} introduced the \emph{Laplacian matching polynomial}
\begin{equation}\label{eq:LMdef}
\LM_G(x)=\sum_{M\in \calM(G)}(-1)^{|M|}\prod_{v\in V(G)\setminus V(M)}(x-d_G(v)).
\end{equation}
Independently, Zhang and Chen \cite{ZhangChen} studied the same polynomial under the 
name \emph{average Laplacian polynomial}. Mohammadian \cite{Moham} proved 
that all roots of $\LM_G(x)$ are real and nonnegative. We write them in non-increasing 
order as
\[
\lambda_1(G)\ge \lambda_2(G)\ge \cdots \ge \lambda_n(G)\ge 0.
\]
Wan, Wang, and Mohammadian \cite{WWM} proved that these roots satisfy an interlacing 
relation under edge addition: if $e$ is a non-edge of $G$, then
\begin{equation}\label{eq:interlacing}
\lambda_1(G+e)\ge \lambda_1(G)\ge \lambda_2(G+e)\ge \lambda_2(G)\ge \cdots \ge \lambda_n(G+e)\ge \lambda_n(G).
\end{equation}
Moreover, by Vieta's formulas applied to \eqref{eq:LMdef},
\begin{equation}\label{eq:sumroots}
\sum_{i=1}^n \lambda_i(G)=\sum_{v\in V(G)} d_G(v)=2|E(G)|.
\end{equation}
Hence, whenever a non-edge is added, the sum of Laplacian matching roots increases by exactly $2$.

Following Wang, Cui, and Cioab\u{a} \cite{WCC}, one says that \emph{Laplacian matching root 
integral variation} (LMRIV) occurs in one place if exactly one root increases by $2$, and in 
two places if exactly two roots increase by $1$ while all the remaining roots stay unchanged. 
In \cite{WCC}, it was proved that one-place LMRIV is impossible for connected graphs, 
and that two-place LMRIV is impossible whenever $g(G)/c(G)>7/6$, where $g(G)$ is the girth 
and $c(G)$ is the dimension of the cycle space. The paper concludes with the following conjecture.

\begin{conjecture}[{\cite[Conjecture~3.11]{WCC}}]\label{conj:WCC}
If $G$ is a connected graph and $e$ is a non-edge of $G$, then two-place LMRIV does not occur when $e$ is added.
\end{conjecture}

The purpose of the present paper is to settle this conjecture affirmatively.

\begin{theorem}\label{thm:main}
Let $G$ be a connected graph and let $e$ be a non-edge of $G$. Then two-place LMRIV cannot 
occur when $e$ is added. Equivalently, if $G'=G+e$, then it is impossible that exactly two 
Laplacian matching roots of $G$ increase by $1$ and all the others remain unchanged.
\end{theorem}

Our proof is short. Besides one key structural lemma from \cite{WCC}, it uses two power-sum 
identities for Laplacian matching roots that seem to be of independent interest.

The motivation for studying  root integral variation comes in part from analogous questions for Laplacian eigenvalues, see So \cite{So} and Kirkland \cite{Kirkland} for example.

\section{Two power-sum identities}

This section aims to establish two power-sum identities, which will play a key role in the sequel. For a graph $G$ and $r\geq 1$, define
$p_r(G):=\sum_{i=1}^n \lambda_i(G)^r$. We also write
\[
A_r(G):=\sum_{v\in V(G)} d_G(v)^r,
\qquad
B(G):=\sum_{xy\in E(G)} d_G(x)d_G(y),
\]
\[
C(G):=\sum_{xy\in E(G)}\bigl(d_G(x)^2d_G(y)+d_G(x)d_G(y)^2\bigr).
\]
Henceforth, whenever there is no risk of confusion, we will suppress the argument $G$.

\begin{proposition}\label{prop:powersums}
Let $G$ be a graph. Then
\[
p_4(G)=A_4+4A_3+2A_2+4B-2|E(G)|,
\]
and
\[
p_5(G)=A_5+5A_4+5A_3-5A_2+5C+10B.
\]
\end{proposition}

\begin{proof}
Write $\LM_G(x)=x^n-b_1x^{n-1}+b_2x^{n-2}-b_3x^{n-3}+b_4x^{n-4}-b_5x^{n-5}+\cdots +(-1)^nb_n$.

Denote by $E_r$ the $r$-th elementary symmetric polynomial in the degree sequence \\
$\{d(v):v\in V(G)\}$. By expanding \eqref{eq:LMdef}, one sees that only matchings of 
size at most $2$ contribute to $b_j$ for $1\le j\le 5$. A direct coefficient computation yields
\begin{align*}
b_1 &= E_1,\\
b_2 &= E_2-|E(G)|,\\ 
b_3 &= E_3-\sum_{xy\in E(G)} \bigl(E_1-d(x)-d(y)\bigr),\\
b_4 & = E_4-\sum_{xy\in E(G)}\bigl(E_2-E_1(d(x)+d(y))
+d(x)^2+d(y)^2+d(x)d(y)\bigr) + \phi_2(G), \\
b_5 & = E_5-\sum_{xy\in E(G)} \!\bigl(E_3-E_2(d(x)+d(y))
+E_1\bigl(d(x)^2+d(x)d(y)+d(y)^2\bigr)\bigr) \\
& \quad+\sum_{xy\in E(G)}\! \bigl(d(x)^3+d(x)^2d(y)+d(x)d(y)^2+d(y)^3\bigr) 
+\sum_{M\in \calM_2(G)}\!\! \bigg(E_1-\sum_{z\in V(M)} d(z)\bigg),
\end{align*}
where $\calM_2(G)$ is the set of $2$-matchings of $G$.

Now, by the definition of $A_r$, we have
\[
A_2 = \sum_{xy\in E(G)}\! \bigl(d(x)+d(y)\bigr),~~~ 
A_3 = \sum_{xy\in E(G)}\! \bigl(d(x)^2+d(y)^2\bigr),~~~
A_4 = \sum_{xy\in E(G)}\! \bigl(d(x)^3+d(y)^3\bigr), 
\]
and
\begin{align*}
\sum_{M\in \calM_2(G)} \sum_{z\in V(M)} d(z)
& = \sum_{xy\in E(G)} \bigl(d(x) + d(y)\bigr) \bigl(|E(G)|-d(x)-d(y)+1\bigr) \\
& = (|E(G)|+1) A_2 - A_3 - 2B.
\end{align*}
Set $m:=|E(G)|$. Using the identities above, together with
\[
\phi_2(G)=\binom{m}{2}-\sum_{v\in V(G)}\binom{d(v)}{2}
=\frac{m^2+m-A_2}{2},
\]
the coefficient formulas become 
\begin{align*}
b_1 &= E_1,\\ 
b_2 &=E_2-m,\\ 
b_3 &=E_3-mE_1+A_2,\\
b_4 & = E_4-mE_2+E_1A_2-A_3-B+\phi_2(G), \\
b_5 & = E_5-mE_3+E_2A_2-E_1(A_3+B)+A_4+C \\
& \quad +E_1\phi_2(G)-(m+1)A_2+A_3+2B.
\end{align*}

Now let $\lambda_1,\dots,\lambda_n$ be the roots of $\LM_G(x)$. 
Newton's identities imply that
\begin{align*}
p_1 &= b_1,\\ 
p_2 &=b_1^2-2b_2,\\ 
p_3 &=b_1^3-3b_1b_2+3b_3,\\
p_4 & = b_1^4-4b_1^2b_2+2b_2^2+4b_1b_3-4b_4, \\
p_5 & = b_1^5-5b_1^3b_2+5b_1b_2^2+5b_1^2b_3-5b_2b_3-5b_1b_4+5b_5.
\end{align*}

Substituting the expressions for $b_1,b_2,b_3,b_4$ into the formula for $p_4$, we get that
\begin{align*}
p_4 & = E_1^4-4E_1^2(E_2-m)+2(E_2-m)^2+4E_1(E_3-mE_1+A_2) \\
& \quad -4\bigl(E_4-mE_2+E_1A_2-A_3-B+\phi_2(G)\bigr) \\
& = E_1^4-4E_1^2E_2+2E_2^2+4E_1E_3-4E_4+4A_3+4B+2m^2-4\phi_2(G).
\end{align*}
We now apply Newton's identities to the degree sequence
$d(v_1),\dots,d(v_n)$, whose elementary symmetric polynomials are precisely
$E_1,E_2,\dots, E_n$, and we get that 
\begin{equation*}
A_4=E_1^4-4E_1^2E_2+2E_2^2+4E_1E_3-4E_4
\end{equation*}
and 
\[
p_4 = A_4+4A_3+4B+2m^2-4\phi_2(G) 
= A_4+4A_3+2A_2+4B-2m.
\]

Similarly, substituting the expressions for $b_1,\dots,b_5$ into the formula for $p_5$ gives
\begin{align*}
p_5 & = E_1^5-5E_1^3(E_2-m)+5E_1(E_2-m)^2+5E_1^2(E_3-mE_1+A_2) \\
& \quad -5(E_2-m)(E_3-mE_1+A_2) 
-5E_1\bigl(E_4-mE_2+E_1A_2-A_3-B+\phi_2(G)\bigr) \\
& \quad +5\bigl(E_5-mE_3+E_2A_2-E_1(A_3+B)+A_4+C \\
& \quad +E_1\phi_2(G)-(m+1)A_2+A_3+2B\bigr) \\
& = E_1^5-5E_1^3E_2+5E_1E_2^2+5E_1^2E_3-5E_2E_3-5E_1E_4+5E_5 \\
& \quad +5A_4+5A_3-5A_2+10B+5C.
\end{align*}
Applying Newton's identities to the degree sequence again, we have
\[
A_5=E_1^5-5E_1^3E_2+5E_1E_2^2+5E_1^2E_3-5E_2E_3-5E_1E_4+5E_5.
\]
Therefore, $p_5=A_5+5A_4+5A_3-5A_2+10B+5C$.
This completes the proof.
\end{proof}

The next lemma is the key input from \cite{WCC}.

\begin{lemma}[{\cite[Theorem~3.2]{WCC}}]\label{lem:WCC}
Assume that two-place LMRIV occurs to a connected graph $G$ by adding a non-edge $uv$. If $a=d_G(u), b=d_G(v)$,
and the changed Laplacian matching roots of $G$ are $\lambda_1(G)$ and $\lambda_k(G)$, then
\begin{equation}\label{eq:sumprod}
\lambda_1(G)+\lambda_k(G)=a+b+1, \text{ and  }
\lambda_1(G)\lambda_k(G)=ab.
\end{equation}
\end{lemma}

\section{Proof of the conjecture}

We are ready to prove Theorem~\ref{thm:main}.

\begin{proof}[Proof of Theorem~\ref{thm:main}]
Assume, to the contrary, that two-place LMRIV occurs to a connected graph $G$ when the non-edge $uv$ is added. Set $G':=G+uv$, $a:=d_G(u)$, and $b:=d_G(v)$.
Since $G$ is connected and $uv\notin E(G)$, we have $a,b\ge 1$. Let the changed roots of $G$ be $\lambda_1:=\lambda_1(G)$, $\lambda_k:=\lambda_k(G)$, so that
$\lambda_1(G')=\lambda_1+1$, $\lambda_k(G')=\lambda_k+1$,
while all the other Laplacian matching roots remain unchanged.

For a vertex $w\in V(G)$, define
\[
S_w:=\sum_{x\in N_G(w)} d_G(x), \quad 
T_w:=\sum_{x\in N_G(w)} d_G(x)^2.
\]
We need the following two claims.

\begin{claim}\label{claim-1}
$S_u+S_v=2ab+a+b$.
\end{claim}

\begin{proof}
For convenience, for $r\geq 1$, write 
\[
\Delta A_r :=A_r(G')-A_r(G),\qquad 
\Delta B :=B(G')-B(G), \qquad
\Delta C :=C(G')-C(G).
\]
Since only the degrees of $u$ and $v$ change when the edge $uv$ is added, we see
\begin{align}\label{eq:DeltaAr}
\Delta A_r & =(a+1)^r-a^r+(b+1)^r-b^r \quad (r\ge 1) \\
\Delta B & = S_u+S_v+(a+1)(b+1). \label{eq:DeltaB}
\end{align}
Using Proposition~\ref{prop:powersums}, we obtain
\[
p_4(G')-p_4(G)=\Delta A_4+4\Delta A_3+2\Delta A_2+4\Delta B-2.
\]
Substituting \eqref{eq:DeltaAr} and \eqref{eq:DeltaB} into the above equation, and then expanding, gives
\begin{align}
p_4(G')-p_4(G)
& = 4(S_u+S_v)+4a^3+18a^2+4ab+24a \notag \\
&\quad +4b^3+18b^2+24b+16.
\label{eq:delta-p4-comb-expanded}
\end{align}

On the other hand, only two roots change, so
\begin{equation}\label{eq:delta-p4-roots}
p_4(G')-p_4(G)
=(\lambda_1+1)^4-\lambda_1^4+(\lambda_k+1)^4-\lambda_k^4
=4(\lambda_1^3+\lambda_k^3)+6(\lambda_1^2+\lambda_k^2)+4(\lambda_1+\lambda_k)+2.
\end{equation}
By Lemma~\ref{lem:WCC}, we deduce that
\begin{align}\label{eq:sumprod-again}
\lambda_1+\lambda_k & = a+b+1, \qquad \lambda_1\lambda_k=ab, \\
\lambda_1^2+\lambda_k^2
& = (\lambda_1+\lambda_k)^2-2\lambda_1\lambda_k = a^2+b^2+2a+2b+1, \label{eq:squaresum} \\
\lambda_1^3+\lambda_k^3
& = (\lambda_1+\lambda_k)^3-3\lambda_1\lambda_k(\lambda_1+\lambda_k) \label{eq:cubesum} \\ 
& = a^3+b^3+3a^2+3b^2+3ab+3a+3b+1. \nonumber
\end{align}
Substituting \eqref{eq:sumprod-again}, \eqref{eq:squaresum}, and \eqref{eq:cubesum} into \eqref{eq:delta-p4-roots}, we get
\[
p_4(G')-p_4(G) = 4a^3+18a^2+12ab+28a + 4b^3+18b^2+28b+16.
\]
Comparing with \eqref{eq:delta-p4-comb-expanded}, we obtain $S_u+S_v=2ab+a+b$.
\end{proof}

\begin{claim}\label{claim-2}
$(2a {+} 3)S_u {+} (2b {+} 3)S_v {+} T_u {+} T_v = 3a^2b {+} 3ab^2 {+} 8ab {+} 2a^2 {+} 2b^2 {+} 4a {+} 4b$.
\end{claim}

\begin{proof}
For $\Delta C$, note that for each edge $ux$ with $x\in N_G(u)$, the quantity
$d_G(u)^2d_G(x)+d_G(u)d_G(x)^2$ increases by
\[
(a+1)^2d_G(x)+(a+1)d_G(x)^2-a^2d_G(x)-ad_G(x)^2
=(2a+1)d_G(x)+d_G(x)^2.
\]
Summing over all $x\in N_G(u)$ gives the contribution $(2a+1)S_u+T_u$.
Similarly, the edges incident with $v$ contribute $(2b+1)S_v+T_v$.
Finally, the new edge $uv$ contributes
\[
(a+1)^2(b+1)+(a+1)(b+1)^2=(a+1)(b+1)(a+b+2).
\]
Hence,
\begin{equation}\label{eq:DeltaC}
\Delta C=(2a+1)S_u+T_u+(2b+1)S_v+T_v+(a+1)(b+1)(a+b+2).
\end{equation}

Using Proposition~\ref{prop:powersums} again, we have
\[
p_5(G')-p_5(G)=\Delta A_5+5\Delta A_4+5\Delta A_3-5\Delta A_2+5\Delta C+10\Delta B.
\]
Substituting \eqref{eq:DeltaAr}, \eqref{eq:DeltaB}, and \eqref{eq:DeltaC} 
into the above equation and simplifying, we obtain
\begin{equation}\label{eq:delta-p5-comb-expanded}
\begin{split}
p_5(G')-p_5(G)
& = 10aS_u+15S_u+10bS_v+15S_v+5T_u+5T_v \\
& \quad+5a^4+30a^3+5a^2b+60a^2+5ab^2+30ab+55a \\
& \quad+5b^4+30b^3+60b^2+55b+32.
\end{split}
\end{equation}

Since only $\lambda_1$ and $\lambda_k$ change,
\begin{align*}
p_5(G')-p_5(G)
& =(\lambda_1+1)^5-\lambda_1^5+(\lambda_k+1)^5-\lambda_k^5\\
& =5(\lambda_1^4+\lambda_k^4)+10(\lambda_1^3+\lambda_k^3)
+10(\lambda_1^2+\lambda_k^2)+5(\lambda_1+\lambda_k)+2.
\end{align*}
Using \eqref{eq:sumprod-again}, \eqref{eq:squaresum}, and \eqref{eq:cubesum} once more, together with
\begin{align*}
\lambda_1^4+\lambda_k^4
& =\bigl(\lambda_1^2+\lambda_k^2\bigr)^2-2(\lambda_1\lambda_k)^2 \\
& =a^4+b^4+4a^3+4b^3+4a^2b+4ab^2+6a^2+6b^2+8ab+4a+4b+1,
\end{align*}
we obtain
\begin{align}
p_5(G')-p_5(G)
& = 5a^4+30a^3+20a^2b+70a^2+20ab^2+70ab+75a \notag\\
& \quad+5b^4+30b^3+70b^2+75b+32.
\label{eq:delta-p5-roots-expanded}
\end{align}
Comparing \eqref{eq:delta-p5-comb-expanded} and \eqref{eq:delta-p5-roots-expanded}, we get the desired result.
\end{proof}

Finally, we will get a contradiction via Claim \ref{claim-1} and Claim \ref{claim-2}. To this end, by Cauchy's inequality, we have
$T_u\geq S_u^2/a$ and $T_v\geq S_v^2/b$.
Set for short, 
\[
R:=3a^2b+3ab^2+8ab+2a^2+2b^2+4a+4b.
\]
By Claim \ref{claim-1}, we may write $S_u=x$, $S_v=2ab+a+b-x$
for some real number $x$. Then Claim \ref{claim-2} implies
\begin{align*}
R
&\ge \frac{x^2}{a}+\frac{(2ab+a+b-x)^2}{b}+(2a+3)x+(2b+3)(2ab+a+b-x)\\
&=R+2ab+\frac{a+b}{ab}\bigl(x-a(b+1)\bigr)^2.
\end{align*}
Since $a,b\ge 1$, the right-hand side is strictly larger than $R$, a contradiction.
This contradiction shows that two-place LMRIV cannot occur. The proof is complete.
\end{proof}

Let $F\subseteq E(G^c)$ be nonempty, where $G^c$ is the complement of $G$. Set $H=G+F$. 
We say that \emph{multi-edge Laplacian matching root integral variation} for $G$ occurs when there exist integers
\[
t_1,\dots,t_n\in \mathbb Z_{\ge 0},\qquad t_1+\cdots+t_n=2|F|,
\]
such that the multiset of Laplacian matching roots of $H$ is obtained from that of $G$ by adding these integers, that is,
\[
\{\lambda_1(H),\dots,\lambda_n(H)\}
=
\{\lambda_1(G)+t_1,\dots,\lambda_n(G)+t_n\}
\]
as multisets.

When $|F|=1$, the only possible integral patterns are $(2,0)$ and $(1,1)$, corresponding 
to the classical one-place and two-place LMRIV considered in the paper.

We finish the paper with the conjecture below which has been verified by computer for $n \leq 9$.
\begin{conjecture}
Let $G$ be a connected graph. Then multi-edge Laplacian matching root integral variation does not occur for $G$.
\end{conjecture}

\end{document}